\documentclass[11pt,reqno]{amsart}
\usepackage{amsthm,amssymb,amsfonts}
\usepackage[nosumlimits]{mathtools}
\usepackage{mathrsfs}
\usepackage{graphicx}
\usepackage{accents}
\usepackage{algorithm}
\usepackage{algpseudocode}
\usepackage{tikz-cd}
\usepackage{enumerate}
\usepackage{quiver}
\usepackage{stmaryrd}
\usepackage{enumitem}
\usepackage{subcaption}
\usepackage{adjustbox}
\usepackage[margin=1in]{geometry}

\usepackage{hyperref,xcolor}
\hypersetup{
    colorlinks,
    linkcolor={red!50!black},
    citecolor={blue!50!black},
    urlcolor={blue!80!black}
}

\usepackage{thmtools}
\usepackage{thm-restate}

\let\oldlhd\unlhd
\renewcommand*{\unlhd}{\mathrel{\mskip0.1mu \oldlhd \mskip0.1mu}}

\theoremstyle{plain}
\newtheorem{theorem}{Theorem}[section]

\newtheorem{lemma}[theorem]{Lemma}

\theoremstyle{definition}
\newtheorem{definition}[theorem]{Definition}

\newtheorem{question}[theorem]{Question}

\DeclareMathOperator{\AR}{AR}
\DeclareMathOperator{\GR}{GR}

\DeclareMathOperator{\codim}{codim}

\DeclareMathOperator{\str}{str}

\DeclareMathOperator{\rank}{rank}

\newcommand{\F}{\mathbb F}
\newcommand{\sr}{\operatorname{sr}}

\begin{document}
\title{Lower bounds on the strength of the determinant}
\date{}
\author[Q-Y.~Chen]{Qiyuan~Chen}
\address{State Key Laboratory of Mathematical Sciences, Academy of Mathematics and Systems Science, Chinese Academy of Sciences, Beijing 100190, China}
\email{chenqiyuan@amss.ac.cn}
\author[Y-H.~Zhao]{Yuhao Zhao}
\address{School of Mathematical Sciences, University of Science and Technology of China, Hefei 230026, China}
\email{zhaoyh21@mail.ustc.edu.cn}
\begin{abstract}
We establish new lower bounds for the strength and partition rank of the determinant.  For every prime $p$, we prove the exact identity
\[
\operatorname{str}(\mathrm{det}_p)=p.
\]
A weak monotonicity argument, combined with a bound for gaps between consecutive primes, then gives $\operatorname{str}(\mathrm{det}_n)\ge (1-o(1))n^{0.475}$ for sufficiently large $n$.  Since the Birch rank of $\mathrm{det}_n$ is always $4$, this gives the first explicit family showing that the dependence on the degree in bounds for strength in terms of Birch rank is unavoidable.  Viewing $\mathrm{det}_n$ as an $n$-linear form in its columns, we also prove that its partition rank is at least the largest prime not exceeding $n$.  Consequently,
\[
n-n^{0.525}\le \operatorname{prk}(\mathrm{det}_n)\le n
\]
for all sufficiently large $n$, and hence the partition rank of the determinant is $n-o(n)$.  The proof introduces an intersection-theoretic method for lower-bounding strength: a short strength decomposition produces a nowhere-vanishing section of a split vector bundle on the complement of the determinantal hypersurface, while a nonzero top Chern class in the Chow ring of $\mathrm{PGL}_n$ obstructs such a section.
\end{abstract}
\maketitle


\section{Introduction}

The strength of a homogeneous polynomial and the partition rank of a multilinear form measure the shortest way of writing the given object as a sum of products of lower-degree objects.  These notions arise in commutative algebra \cite{ananyan2020small}, analytic number theory \cite{schmidt1985density}, additive combinatorics \cite{gowers2011linear}, and the polynomial method \cite{Naslund20,tao2016slice}. In this paper, we study them for one of the most classical explicit polynomials, the determinant. Besides obtaining nearly sharp lower bounds, we develop a new intersection-theoretic mechanism for certifying that a proposed decomposition cannot exist, which is of independent interest.

Our first motivation comes from the relation between strength and the codimension of singularities.  Ananyan and Hochster proved in \cite{ananyan2020small} that, for a form of fixed degree $d$, sufficiently large strength forces the singular locus to have a large codimension. This codimension is often referred to the Birch rank of the form.  Equivalently, the strength of a degree-$d$ form is bounded in terms of its Birch rank and $d$.  A substantial body of later work sought effective and quantitatively stronger versions of this principle; see, among others, \cite{adiprasito2021schmidt,cohen2021structure,cohen2023partition,kazhdan2024schmidt}.  The strongest explicit estimate relevant here is due to Kazhdan, Lampert, and Polishchuk: under the assumption that the characteristic does not divide $d(d-1)$, their Schmidt rank is at most $(d-1)$ times the codimension of the singular locus; see \cite[Theorem~A.2]{kazhdan2024schmidt}. More recently, several authors established bounds of strength (resp. partition rank) in the Birch rank (resp. geometric rank/analytic rank) over broad classes of fields, with a constant depending on the degree (sometimes on the field) \cite{moshkovitz2022quasi,baily2024strength,chenye2026geometry}.

Before the present work, however, no explicit family showed that the dependence on $d$ in such a comparison is genuinely necessary. All bounds obtained up to now depend on $d$ as a linear factor.  It is important here not to confuse this question with the necessity of a degree parameter in Stillman's conjecture itself.  The classical examples showing that Stillman-type bounds must depend on the degrees concern projective dimensions of ideals.  Combining those examples with the Ananyan--Hochster regularization argument does not produce a family of individual forms having bounded Birch rank and unbounded strength: the mechanism used to prove Stillman's conjecture (or to obtain a small subalgebra) is not a converse that turns large projective dimension into such a rank separation.  Our determinant examples give the first direct obstruction result.  Indeed, for every prime $p$ we prove that the strength of $\det_p$ has the exact value
\[
\str(\mathrm{det}_p)=p,
\]
whereas the Birch rank of $\mathrm{det}_p$ is equal to $4$. Thus, any universal inequality of the form
\[
\str(f)\le C(d)\operatorname{Brk}(f)
\]
must satisfy $C(p)\ge p/4$ for infinitely many degrees $p$.  In particular, a linear dependence on the degree is unavoidable up to an absolute constant.

Our second motivation is combinatorial.  Partition rank is one of the central algebraic quantities underlying the polynomial method and its tensor formulations \cite{naslund2017upper,Naslund20}.  Yet proving a strong lower bound for the partition rank of a general tensor is notoriously difficult.  Lampert and Moshkovitz \cite{lampert2026determinant} recently initiated a systematic study of the slice rank and partition rank of the determinant from this viewpoint.  They proved the logarithmic lower bound
\[
\operatorname{prk}(\mathrm{det}_n)\ge \log_2 n+1
\]
for the partition rank of the determinant, and explained why the previously available approaches could not improve it beyond logarithmic order. Moreover, they asked for the asymptotic behavior of $\operatorname{prk}(\mathrm{det}_n)$.  They also conjectured that $\str(\mathrm{det}_n)$ tends to infinity when $n$ tends to infinity.

The determinant is particularly significant for the comparison between algebraic structure and analytic or geometric randomness.  Over finite fields, partition rank is bounded below by analytic rank, and the reverse comparison, with a constant depending only on the tensor order, is a central theme in additive combinatorics; there are several works on this theme \cite{milicevic2019polynomial,janzer2019polynomial,cohen2023partition,moshkovitz2022quasi}.  Geometric rank, introduced by Kopparty, Moshkovitz and Zuiddam in \cite{kopparty2020geometric} is the algebro-geometric analogue of analytic rank; recent work \cite{chen2024stability,moshkovitz2024uniform,baily2024strength} shows that analytic rank and geometric rank are equivalent up to constants depending only on the tensor order, uniformly in the finite field.  Lampert and Moshkovitz \cite{lampert2026determinant} formalized the possible dependence on the order by considering
\[
A(d):=\sup_T\frac{\operatorname{prk}(T)}{\lceil\AR(T)\rceil},
\qquad
G(d):=\sup_T\frac{\operatorname{prk}(T)}{\GR(T)},
\]
and asked whether either quantity grows faster than $\log d$.

We now describe our main results towards the above questions.  Throughout, for $n\ge 2$ let $p(n)$ denote the largest prime not exceeding $n$.  First, for every prime $p$ we determine the strength exactly:
\[
\str(\mathrm{det}_p)=p.
\]
For arbitrary $n$, a weak monotonicity property for determinant strength, together with the prime-gap estimate recalled below, yields
\[
\str(\mathrm{det}_n)\ge \frac{p(n)}{n-p(n)+1}\ge (1-o(1))n^{0.475}
\]
for all sufficiently large $n$.  Since $\operatorname{Brk}(\mathrm{det}_n)=4$, this produces the degree-dependent separation discussed above and, in particular, confirms the conjecture \cite[Conjecture~2]{lampert2026determinant} of Lampert and Moshkovitz that the strength of the determinant tends to infinity.

For partition rank we prove the stronger monotonicity needed to obtain
\[
\operatorname{prk}(\mathrm{det}_n)\ge p(n).
\]
Consequently, for all sufficiently large $n$,
\[
n-n^{0.525}\le \operatorname{prk}(\mathrm{det}_n)\le n,
\]
where the upper bound is from the Laplace expansion.  Thus, we have
\[
\operatorname{prk}(\mathrm{det}_n)=n-o(n),
\]
which determines its asymptotic value and answers \cite[Problem~1]{lampert2026determinant}.  Moreover, since $\GR(\mathrm{det}_d)=2$, and over every finite field one has $\lceil\AR(\mathrm{det}_d)\rceil=2$ \cite{lampert2026determinant}, it follows that
\[
A(d),G(d)\ge \frac{p(d)}{2}=\frac d2-o(d).
\]
This gives a positive answer, in a substantially stronger form, to their problem \cite[Problem~3]{lampert2026determinant} asking for a super-logarithmic separation between partition rank and analytic or geometric rank.

The main new ingredient in our proof is an intersection-theoretic interpretation of strength lower bounds.  A decomposition placing a form $f$ in an ideal generated by $r$ lower-degree forms gives a global section of a direct sum of line bundles on
\[
\mathbb P(V)\setminus V_+(f)
\]
that is nowhere zero.  A nonvanishing top Chern class therefore rules out such a decomposition.  For $f=\mathrm{det}_p$, the complement is $\mathrm{PGL}_p$, and the relevant Chow ring is
\[
CH^*(\mathrm{PGL}_p)\cong \mathbb Z[x]/(px,x^p).
\]
If a strength decomposition with at most $p-1$ summands existed, its associated split bundle would have top Chern class
\[
\left(\prod_{i=1}^{p-1}d_i\right)x^{p-1},
\qquad 1\le d_i\le p-1,
\]
which is nonzero in this ring and leads to a contradiction.  This use of intersection theory to detect the impossibility of a nowhere-vanishing section provides a new tool for explicit lower bounds, distinct from the restriction, derivative, and combinatorial methods previously used for lower bounding the partition rank.

The rest of this paper is organized as follows.  In Section~\ref{sec:pre}, we first review the rank notions and the necessary materials on vector bundles, Chow groups, and Chern classes.  We then compute the part of the Chow ring of $\mathrm{PGL}_n$ needed for the obstruction in Section~\ref{sec:intersection-PGL}.  Next, we prove the exact result for prime sizes and derive a polynomial lower bound for general determinant strength using Hasse derivatives and prime gaps in Section~\ref{sec:exact}.  In Section~\ref{sec:monotonicity}, we establish the monotonicity of the partition rank of the determinant and deduce its asymptotic value. In Section~\ref{sec:str-vs-par}, we further discuss the relationship between strength and partition rank. Section~\ref{sec:discussion} contains some final discussions and open problems.

\section{Preliminaries and notation}\label{sec:pre}
Throughout this paper, unless otherwise stated, $K$ denotes an algebraically closed field.  All vector spaces are finite-dimensional over $K$.  For an integer $d\geq 1$, we write $[d]=\{1,\ldots,d\}$.  If $V$ is a vector space, then $V^*$ denotes its dual.  For a polynomial ring $S=K[x_1,\ldots,x_N]$, the homogeneous degree-$e$ part of $S$ is denoted by $S_e$.
\subsection{Rank notions for tensors and homogeneous polynomials}
We first recall several rank notions of tensors and forms that will be used later.
In this part, let us briefly introduce the definitions of partition rank and geometric rank of a tensor, and the definitions of strength and Birch rank of a homogeneous polynomial. 
\begin{definition}
Let $d\ge 2$, and let
\[
  T\in V_1^*\otimes\cdots\otimes V_d^*
\]
be a $d$-linear form.  A tensor $T$ has \emph{partition rank one} if there is a nonempty proper subset $I\subsetneq [d]$ and multilinear forms
\[
  A\in \bigotimes_{i\in I}V_i^*,
  \qquad
  B\in \bigotimes_{i\notin I}V_i^*
\]
such that
\[
  T(v_1,\ldots,v_d)=A((v_i)_{i\in I})B((v_i)_{i\notin I}).
\]
The \emph{partition rank} of $T$, denoted $\operatorname{prk}_K(T)$, is the smallest integer $r$ such that $T$ is a sum of $r$ partition-rank-one tensors. 
The \textit{slice rank} $\sr_{K}(T)$ of $T$ is the least $r$ for which $T(v_1,\ldots,v_d)=\sum_{a=1}^{r}A(v_{i_a})B((v_i)_{i\ne i_a})$, where $i_a\in [d]$, $A\in V_{i_a}^*$, and $B\in\bigotimes_{i\ne i_a}V_i^*$.

Following the standard convention in the tensor-rank literature, we view $T$ as a multilinear map
\[
  V_1\times\cdots\times V_{d-1}\longrightarrow V_d^*,
  \qquad
  (v_1,\ldots,v_{d-1})\mapsto T(v_1,\ldots,v_{d-1},-).
\]
Define the associated multilinear variety
\[
  Z_T:=\{(v_1,\ldots,v_{d-1})\in V_1\times\cdots\times V_{d-1}:
  T(v_1,\ldots,v_{d-1},-)=0\in V_d^*\}.
\]
The \emph{geometric rank} of $T$ is
\[
  \operatorname{GR}_K(T):=
  \operatorname{codim}_{V_1\times\cdots\times V_{d-1}} Z_T.
\]
\end{definition}
\begin{definition}
    Let $f\in S_d$ be a homogeneous polynomial of degree $d\geq 2$.  The \emph{strength} of $f$ over $K$, denoted $\operatorname{str}_K(f)$, is the smallest integer $r$ for which one can write
\[
  f=\sum_{i=1}^r g_i h_i,
\]
where $g_i,h_i\in S$ are homogeneous polynomials satisfying
\[
  0<\deg g_i,\deg h_i<d.
\]
The \textit{polynomial slice rank}
$\sr_K^{\rm pol}(f)$ of $f$ is the least $r$ such that $f=\sum_{i=1}^r \ell_i g_i$, where $\ell_i, g_i\in S$ are homogeneous polynomials with
$\deg \ell_i=1$ and $\deg g_i=d-1$.

The \emph{Birch rank} of $f$ is the codimension of its affine singular locus:
\[
  \operatorname{Brk}_K(f):=
  \operatorname{codim}_{\mathbb A_K^N}
  V\left(\frac{\partial f}{\partial x_1},\ldots,\frac{\partial f}{\partial x_N}\right).
\]
Here $V(\cdot)$ denotes the affine common zero locus.
\end{definition}
When the ground field is fixed or clear from context, we omit the subscript \(K\) from these ranks.

\subsection{Preliminaries on algebraic geometry}
In this part, we recall some elementary notation and constructions from algebraic geometry.  A standard reference is  \cite[Chapters I--II]{hartshorne2013algebraic}.
\subsubsection{Projective varieties.}
 If $S=K[x_0,\ldots,x_N]$ is graded and $I\subset S$ is a homogeneous ideal, then
\[
  V_+(I)=\{[a_0:\cdots:a_N]\in \mathbb P_K^N: f(a_0,\ldots,a_N)=0\text{ for all homogeneous }f\in I\}
\]
is the corresponding projective zero locus.  For a homogeneous element $f\in S$, the standard principal open subset of $\operatorname{Proj}S$ is denoted by
\[
  D_+(f):=\{\mathfrak p\in \operatorname{Proj}S: f\notin \mathfrak p\}.
\]
\subsubsection{Sheaves on a variety}
A variety is a topological space together with a sheaf of rings, called the structure sheaf, which records its regular functions. Given a variety $X$, we often use $\mathcal{O}_X$ to denote its structure sheaf. 

For a projective variety $X=\operatorname{Proj}R$, where $R$ is a finitely generated graded $K$-algebra, one has
\[
  \mathcal O_X(D_+(f))=R_{(f),0}.
\]

An $\mathcal O_X$-module is a sheaf $\mathcal F$ such that each $\mathcal F(U)$ is a module over $\mathcal O_X(U)$, compatibly with restrictions. A locally free sheaf of finite rank is equivalently an algebraic vector
bundle on \(X\). Under this correspondence, invertible sheaves are the same
as line bundles. In this paper we will freely move between the two languages.

The isomorphism classes of line bundles on a variety $X$ form a group (the multiplication is defined by the tensor product of line bundles), which is called the Picard group of $X$ and is often denoted by $\mathrm{Pic}(X)$. 

A special family of line bundles over a projective variety is Serre's twisting sheaf. 
Given a graded ring $R$, if $X=\operatorname{Proj}R$, then $\mathcal O_X(m)$ is the sheaf associated with the shifted graded module $R(m)$. A homogeneous polynomial $F\in K[x_0,
\ldots,x_N]_m$ of degree $m$ defines a global section of $\mathcal O_{\mathbb P^N}(m)$.  On the affine chart $D_+(x_i)$, this section is represented by the ordinary regular function \(\frac{F}{x_i^m}\) in the coordinates $x_j/x_i$.  Thus a section of $\mathcal O(m)$ should be thought of as a homogeneous degree-$m$ function. 

Finally, we illustrate  the meaning of the zero locus of a global section of a line bundle. If $L$ is a line bundle on a variety $X$ and $s\in H^0(X,L)$ is a global section, then the zero locus of $s$ is defined locally as follows.  On an open set $U$ where $L$ is a trivial $\mathcal{O}_{X}$-module of rank 1 with basis $e$, write
\[
  s|_U=f_U e,
  \qquad f_U\in \mathcal {O}_X(U).
\]
Then the zero locus of $s$ restricted to $U$ is the ordinary zero locus of the regular function $f_U$.  The local definitions can be glued by local chart.  This also illustrates the meaning of the common zero locus of a section of a vector bundle.
\subsubsection{Complete flag variety and tautological bundle}
\begin{definition}[Complete flag variety]
    A \emph{complete flag} in $V$ is a strictly increasing chain of subspaces
    \[
    V_\bullet:\quad 0 = V_0 \subset V_1 \subset V_2 \subset \cdots \subset V_n = V,
    \]
    satisfying $\dim V_i = i$ for all $1\leq i\leq n$.
    
    The collection of all such complete flags admits a natural structure of a smooth projective algebraic variety, called the \emph{complete flag variety} of $V$, denoted by $\operatorname{Fl}(V)$ or $\operatorname{Fl}(n)$.
    
    As a homogeneous space of the general linear group, the flag variety admits the presentation
    \[
    \operatorname{Fl}(V) \cong \operatorname{GL}(V) / B,
    \]
    where $B \subset \operatorname{GL}(V)$ is the Borel subgroup of upper triangular matrices preserving a fixed standard flag.
\end{definition}

\begin{definition}[Tautological subbundles on the flag variety]
    On the complete flag variety $\operatorname{Fl}(V)$, there is a canonical family of algebraic vector bundles $\{S_i\}_{i=1}^n$, called the \emph{tautological subbundles}.
    
    For each index $i \in \{1,2,\dots,n\}$, $S_i$ is a vector bundle of rank $i$, whose fiber at any point $V_\bullet = (V_1\subset\cdots\subset V_n) \in \operatorname{Fl}(V)$ is exactly the $i$-th subspace in the flag:
    \[
    \left. S_i \right|_{V_\bullet} = V_i.
    \]
    
    By the inclusion relations of the flag, these subbundles form an increasing chain
    \[
    0 = S_0 \subset S_1 \subset S_2 \subset \cdots \subset S_n = \underline{V},
    \]
    called the \emph{tautological flag}, where $\underline{V} := V \times \operatorname{Fl}(V)$ is the trivial rank-$n$ vector bundle over $\operatorname{Fl}(V)$.
    
    Furthermore, for each $1\leq i\leq n$, the successive quotient
    \[
    L_i := S_i \big/ S_{i-1}
    \]
    is a line bundle on $\operatorname{Fl}(V)$, called a \emph{tautological quotient line bundle}.
\end{definition}
\subsection{Preliminaries on intersection theory}
We next recall part of intersection theory that will be used later.  The standard reference is \cite{fulton2013intersection} or \cite{eisenbud20163264}.

Let $X$ be an algebraic variety over an algebraically closed field.  The group $CH_i(X)$ is the abelian group generated by $i$-dimensional irreducible subvarieties of $X$, modulo rational equivalence.  If $X$ is smooth and of pure dimension $n$, we write
\[
  CH^r(X):=CH_{n-r}(X).
\]
In this case, intersection products make
\[
  CH^*(X)=\bigoplus_{r\geq 0}CH^r(X)
\]
into a graded ring, called the Chow ring.  The grading is by codimension.

Under the assumption that $X$ is smooth, there is a canonical isomorphism from $\mathrm{Pic}(X)$ to $CH^{1}(X)$. It is noticeable to mention that under this assumption, $CH^{1}(X)$ is isomorphic to the group of Cartier divisors of $X$ canonically.
\begin{lemma}[Localization sequence]\label{lem-localization sequence}
Let \( X \) be an algebraic variety, \( Z \subset X \) a closed subvariety, and \( U = X \setminus Z \) the open complement. Denote by \( i \colon Z \hookrightarrow X \) the closed immersion and \( j \colon U \hookrightarrow X \) the open immersion.

Then for every integer \( k \), there exists a right-exact localization sequence of Chow groups:
\[
CH_k(Z) \xrightarrow{i_*} CH_k(X) \xrightarrow{j^*} CH_k(U) \longrightarrow 0,
\]
where:
\begin{itemize}
    \item \( i_* \) is the pushforward induced by the closed immersion \( i \);
    \item \( j^* \) is the pullback induced by the flat open immersion \( j \).
\end{itemize}
\end{lemma}

\begin{definition}[Algebraic definition of Chern class of a vector bundle]\label{def-Chern class}
    Let $E$ be a vector bundle of rank $r$ on a smooth variety $X$.  Its Chern classes are elements
\[
  c_i(E)\in CH^i(X),
  \qquad i\ge 0,
\] while its total Chern class
\[
c(E)=\sum_{i=0}^{r}c_{i}(E)
\]
with $c_0(E)=1$, which have the following properties:
\begin{enumerate}
    \item For any regular map $f:Y\to X$, $c_{i}(f^{*}(E))=f^{*}(c_{i}(E))$.
    \item For any vector bundle $E_1,E_2,E_3$ over $X$ such that \[\begin{tikzcd}
	0 & {E_1} & {E_2} & {E_3} & 0
	\arrow[from=1-1, to=1-2]
	\arrow[from=1-2, to=1-3]
	\arrow[from=1-3, to=1-4]
	\arrow[from=1-4, to=1-5]
\end{tikzcd}\] is exact , we have $c(E_2)=c(E_1)\cdot c(E_3)$.
    \item For any line bundle $L$, $c_{1}(L)=\phi(L)\in CH^{1}(X)$, where $\phi$ denotes the canonical isomorphism from $\mathrm{Pic}(X)$ to $CH^{1}(X)$.
    \item If a vector bundle $F$ has rank $s$, then $c_{i}(F)=0$ for all $i>s$.
\end{enumerate}
\end{definition}
There are two useful computational properties that are straightforward from the definition.
\begin{lemma}\label{lemma-Computational lemma of Chern class}
    \begin{enumerate}
        \item Let $E_{i}$ be line bundles of $X$ ($1\le i\le s$). Then we have the top Chern class  $c_{s}(\bigoplus_{i=1}^{s}E_i)=\prod_{i=1}^{s}c_{1}(E_i)$.
        \item Let $L$ be a line bundle of $X$. Then $c_{1}(L^{\otimes d})=d\cdot c_{1}(L)$.
    \end{enumerate}     
\end{lemma}
\begin{proof}
    The first claim follows directly from (2) in Definition~\ref{def-Chern class}. The second claim follows directly from (3) in Definition~\ref{def-Chern class} and the fact that $\phi$ is a group isomorphism from $\mathrm{Pic}(X)$ to $CH^{1}(X)$.
\end{proof}
\begin{lemma}\label{nonvanishing lemma}
    If $E$ is a rank-$r$ vector bundle of $X$ and $E$ has a global section $S$ that is nowhere zero on $X$, then $c_{r}(E)=0$.
\end{lemma}
\begin{proof}
The nowhere-vanishing section \(S\) gives an exact sequence
\[
0\longrightarrow \mathcal O_X \xrightarrow{S} E\longrightarrow F\longrightarrow 0,
\]
where \(F\) is a vector bundle of rank \(r-1\). Hence
\[
c(E)=c(\mathcal O_X)c(F)=c(F).
\]
Since \(\operatorname{rank} F=r-1\), we have \(c_r(F)=0\). Therefore
\(c_r(E)=0\).
\end{proof}
\section{Intersection theory of $\mathrm{PGL}_n$}\label{sec:intersection-PGL}
In this section, we use Grothendieck's computation of the Chow ring of a
split semisimple linear algebraic group to compute the Chow ring of
\(\operatorname{PGL}_n\). The standard reference is \cite[Remark 2, p. 21]{Grothendieck1958Torsion}. We will use some standard definitions and notation from algebraic groups;
see, for example \cite{Procesi2007LieGroups} or \cite{Milne2017Algebraic}.

We first present Grothendieck's solution. This presentation is due to Mackall in \cite{Mackall2018Chow}.

Let \( G \) be a split semisimple linear algebraic group over an algebraically closed field $\mathbb{F}$, let $T \subset G$  be a split maximal torus, and let \( B \supset T \) be a Borel subgroup of \( G \). Denote by \(\mathbb{G}_m\) the multiplicative group of $\mathbb{F}$ and by  \( \mathcal{X}^*(B) = \operatorname{Hom}(B, \mathbb{G}_m) \) the character group of \( B \). Since the unipotent radical of \( B \) admits no nontrivial characters, there is a natural isomorphism \( \mathcal{X}^*(B) \cong \mathcal{X}^*(T) \).

The quotient \( G/B \) is a smooth projective variety. To every character \( \rho \in \mathcal{X}^*(B) \), one associates a homogeneous line bundle \( L(\rho) \) over \( G/B \), constructed as the quotient
\[
L(\rho) = (G \times \mathbb{A}^1) / B,
\]
where \( B \) acts on \( G \times \mathbb{A}^1 \) according to the rule
\[
b \cdot (g, y) = \left( gb,\ \rho(b)^{-1} y \right).
\]
Taking the first Chern class of these line bundles yields a group homomorphism
\[
\mathcal{X}^*(B) \to \mathrm{Pic}(G/B) = CH^1(G/B) \subset CH^*(G/B),
\]
which extends uniquely to a homomorphism of graded \( \mathbb{Z} \)-algebras
\[
c_G : S^*(\mathcal{X}^*(B)) \to CH^*(G/B).
\]
This map is called the \emph{characteristic homomorphism}.

Since Grothendieck's solution states that the Chow ring of \( G \) can be expressed as a quotient of the Chow ring of $G/B$:
\[
CH^*(G) \cong CH^*(G/B) \big/ I,
\]
where \( I \subset CH^*(G/B) \) is the homogeneous ideal generated by the image of the degree-1 component \( c_G(\mathcal{X}^*(B)) \).
Before our computation of the Chow ring of $\mathrm{PGL}_n$, we need the following well-known lemma on the Chow ring of complete flag variety $\mathrm{Fl}_n$.
\begin{lemma}\label{lem-chow ring of flag variety}
\[
CH^{*}(G/B)=CH^{*}(\mathrm{Fl}_{n})=\mathbb{Z}[x_1,\cdots,x_n]/(e_1,\cdots,e_n),
\] where $e_i$ is $i$-th elementary symmetric function of $x_1,\cdots,x_n$.

Moreover, each \( x_i \) is the first Chern class of the \( i \)-th tautological quotient line bundle on \( \mathrm{Fl}(n) \), of degree \( 1 \), and \( e_k = e_k(x_1,\dots,x_n) \) denotes the \( k \)-th elementary symmetric polynomial in \( x_1,\dots,x_n \).
\end{lemma}
\begin{proof}
    This is by applying \cite[Lemma~5.3]{Fulton1992Flags} to flag bundles over a single point and by choosing tautological bundles to be the universal bundles.
\end{proof}
\begin{lemma}[$\mathbb{Z}$-coefficient Chow ring of $\mathrm{PGL}_n$]\label{lem-Chow ring of PGL_n}
    $CH^*(\mathrm{PGL}_n)\cong \mathbb{Z}[x] \big/ \left( \binom{n}{1}x,\ \binom{n}{2}x^2,\ \dots,\ \binom{n}{n}x^n\right)$.
\end{lemma}
\begin{proof}
    We apply Grothendieck's solution. Let \(Z\subset \operatorname{GL}_n\) be the subgroup of scalar matrices. First, set \( D_n \subset \operatorname{GL}_n \) be the maximal torus consisting of diagonal matrices and \( B_n \subset \operatorname{GL}_n \) the Borel subgroup consisting of upper triangular invertible matrices. 
    Set $G=\operatorname{GL}_n/Z=\operatorname{PGL}_n,T=D_n/Z$, and $B=B_n/Z$. 
Then \(G/B\simeq \operatorname{GL}_n/B_n\) is the complete flag variety of rank \(n\).

    As pointed above, \( \mathcal{X}^*(B) \cong \mathcal{X}^*(T) \). By definition of the quotient \( T = D_n / \mathbb{G}_m \), the character group \( \mathcal{X}^*(T) \) consists of characters of \( D_n \) that restrict to the trivial character on the central \( \mathbb{G}_m \):
\[
\mathcal{X}^*(T) = \left\{ \chi \in \mathcal{X}^*(D_n) \mid \chi|_{\mathbb{G}_m} = 1 \right\}\subset \mathcal{X}^*(D_n).
\]

Let \( \chi_1, \dots, \chi_n \) be the standard basis of \( \mathcal{X}^*(D_n) \), where \( \chi_i(\operatorname{diag}(t_1,\dots,t_n)) = t_i \). The central embedding \( t \mapsto t \cdot I_n \) sends \( t \in \mathbb{G}_m \) to the scalar matrix \(\mathrm{Diag}( t) \). A character \( \chi = \sum_{i=1}^n a_i \chi_i \) restricts trivially to \( \mathbb{G}_m \) if and only if
\[
\sum_{i=1}^n a_i = 0.
\]

Thus, \( \mathcal{X}^*(T) \) is a free abelian group of rank \( n-1 \), generated by the simple roots
\[
\alpha_i = \chi_i - \chi_{i+1}, \quad i = 1, 2, \dots, n-1.
\]
We have 
\[
CH^*(\mathrm{Fl}(n)) \cong \mathbb{Z}[x_1, x_2, \dots, x_n] \big/ (e_1, e_2, \dots, e_n),
\]
where each \( x_i \) is the first Chern class of the \( i \)-th tautological quotient line bundle on \( \mathrm{Fl}(n) \).

Under the characteristic homomorphism \( c_{\mathrm{PGL}_n} : S^*(\mathcal{X}^*(B)) \to CH^*(\mathrm{Fl}(n)) \), each \( \alpha_i \) maps to the difference of Chern classes:
\[
c_{\mathrm{PGL}_n}(\alpha_i) = x_i - x_{i+1} \in CH^1(\mathrm{Fl}(n)).
\]
Then
\[
CH^*(\mathrm{PGL}_n) = CH^*(\mathrm{Fl}(n)) \big/ I,
\]
where \( I \) is the ideal generated by \( \{ x_i - x_{i+1} \mid i = 1, \dots, n-1 \} \).

The relations \( x_1 - x_2 = 0,\ x_2 - x_3 = 0,\ \dots,\ x_{n-1} - x_n = 0 \) force all variables to be equal in the quotient. Set
\[
x = x_1 = x_2 = \cdots = x_n.
\]
Substituting \( x_i = x \) into the elementary symmetric polynomials yields
\[
e_k(\underbrace{x, x, \dots, x}_{n\text{ times}}) = \binom{n}{k} x^k
\]
for each \( k = 1, 2, \dots, n \). The original relations \( e_k = 0 \) in \( CH^*(\mathrm{Fl}(n)) \) therefore become
\[
\binom{n}{k} x^k = 0
\]
in the quotient ring.

It follows that \( CH^*(\mathrm{PGL}_n) \) is generated by the single degree-1 element \( x \), subject to the relations above. This proves
\[
CH^*(\mathrm{PGL}_n) \cong \mathbb{Z}[x] \bigg/ \left( \binom{n}{1}x,\ \binom{n}{2}x^2,\ \dots,\ \binom{n}{n}x^n \right).
\]
\end{proof}
\begin{lemma}\label{lem-Chern class of twisting sheaf}
    Let $\mathcal{L}_{i}\coloneqq \mathcal{O}_{\mathbb{P}(\mathrm{Mat}_n)}(i)|_{\mathrm{PGL}_n}$. Then $c_{1}(\mathcal{L}_{1})$ is a generator of $CH^{1}(\mathrm{PGL}_n)$ (as an abelian group). Moreover, if $n$ is prime, $CH^{*}(\mathrm{PGL}_n)\cong\mathbb{Z}[x]/\langle nx,x^{n}\rangle$ and $c_{1}(\mathcal{L}_{1})$ is a degree-$1$ generator of $CH^{*}(\mathrm{PGL}_n)=\mathbb{Z}[x]/\langle nx,x^{n}\rangle$ as a $\mathbb{Z}$-algebra.
\end{lemma}
\begin{proof}
By Lemma~\ref{lem-Chow ring of PGL_n}, we have $CH^{1}(\mathrm{PGL}_n)=\mathbb{Z}/n\mathbb{Z}$. By Lemma~\ref{lem-localization sequence}, 
\[
CH_{n^2-2}(V_{+}(\mathrm{det}_n)) \xrightarrow{i_*} CH_{n^2-2}(\mathbb{P}(\mathrm{Mat}_n)) \xrightarrow{j^*} CH_{n^2-2}(\mathrm{PGL}_n) \longrightarrow 0.
\]
A well-known fact is that $\mathrm{det}_n$ is irreducible. Hence $CH_{n^2-2}(V_{+}(\mathrm{det}_n))\cong \mathbb{Z}$, $i_*$ maps $[V_{+}(\mathrm{det}_n)]$ to $[V_{+}(\mathrm{det}_n)]$ in $CH_{n^2-2}(\mathbb{P}(\mathrm{Mat}_n))$ and is equal to $$c_{1}(\mathcal{O}_{\mathbb{P}(\mathrm{Mat}_n)}(n))=c_{1}(\mathcal{O}_{\mathbb{P}(\mathrm{Mat}_n)}(1)^{\otimes n})=nc_{1}(\mathcal{O}_{\mathbb{P}(\mathrm{Mat}_n)}(1))$$ by Lemma~\ref{lemma-Computational lemma of Chern class}. Hence\[
CH_{n^2-2}(\mathbb{P}(\mathrm{Mat}_n))/i_{*}(CH_{n^2-2}(V_{+}(\mathrm{det}_n)))\cong \mathbb{Z}/n\mathbb{Z},
\] and the generator is $c_{1}(\mathcal{O}_{\mathbb{P}(\mathrm{Mat}_n)}(1))$. By (1) in Definition~\ref{def-Chern class} we see that \[
j^{*}(c_{1}(\mathcal{O}_{\mathbb{P}(\mathrm{Mat}_n)}(1)))=c_{1}(\mathcal{O}_{\mathbb{P}(\mathrm{Mat}_n)}(1)|_{\mathrm{PGL}_n}).
\]
This implies that $c_{1}(\mathcal{O}_{\mathbb{P}(\mathrm{Mat}_n)}(1)|_{\mathrm{PGL}_n})$ is a generator of $CH^{1}(\mathrm{PGL}_n)$.

If $n$ is a prime number, it is clear that $n|\binom{n}{i}$ when $i<n$, and hence $$CH^{*}(\mathrm{PGL}_n)\cong\mathbb{Z}[x]/\langle nx,x^{n}\rangle.$$ Since $CH^{1}(\mathrm{PGL}_n)\cong\mathbb{Z}/n\mathbb{Z}$ and $c_{1}(\mathcal{O}_{\mathbb{P}(\mathrm{Mat}_n)}(1)|_{\mathrm{PGL}_n})$ is a generator of $CH^{1}(\mathrm{PGL}_n)$, the next claim follows directly.
\end{proof}

\section{Proof of the main theorem}\label{sec:exact}
In this section, we prove lower bounds on $\mathrm{str}(\mathrm{det}_n)$. We first show the following lemma.
\begin{lemma}\label{lem-intepretion}
    Given any positive integer $d\ge 2$ and a degree-$d$ homogeneous polynomial $f$, if for any $1\le d_i\le d-1$, $1\le i\le r$, any global section of \(\bigoplus_{i=1}^{r}(\mathcal{O}_{\mathbb{P}^n}(1)|_{\mathbb{P}^n\backslash V_{+}(f)})^{\otimes d_{i}}\) has vanishing points on $\mathbb{P}^n\backslash V_{+}(f)$, then \[ 
    \mathrm{str}(f)>r.
    \]
\end{lemma}
\begin{proof}
    According to \cite[Lemma~17.16.4]{stacks-project}, restricted map commutes with tensor product on the sheaf of modules. Hence \[
    \mathcal{O}_{\mathbb{P}^n}(e)|_{\mathbb{P}^n\backslash V_{+}(f)}=(\mathcal{O}_{\mathbb{P}^n}(1)|_{\mathbb{P}^n\backslash V_{+}(f)})^{\otimes e}.
    \]
    If $\mathrm{str}(f)\le r$, we may assume that there are forms $f_i$ with degree $1\le d_i<d$ such that $f\in \langle f_1,\cdots,f_r\rangle$. Each $f_i$ defines a global section of $\mathcal{O}_{\mathbb{P}^{n}}(d_i)$ and hence a global section of $\mathcal{O}_{\mathbb{P}^{n}}(d_i)|_{\mathbb{P}^n\backslash V_{+}(f)}$. By Hilbert Nullstellensatz, $V_{+}(f_1,\cdots,f_r)\subset V_{+}(f)$. In other words, there is a global section of $$\bigoplus_{i=1}^{r}\mathcal{O}_{\mathbb{P}^n}(d_i)|_{\mathbb{P}^n\backslash V_{+}(f)})^{\otimes d_{i}}=\bigoplus_{i=1}^{r}(\mathcal{O}_{\mathbb{P}^n}(1)|_{\mathbb{P}^n\backslash V_{+}(f)}$$ defined by $(f_1,\cdots,f_r)$ that is nowhere vanishing on $\mathbb{P}^n\backslash V_{+}(f)$. This is a contradiction.
\end{proof}
Lemma~\ref{lem-intepretion} gives a interpretation of the problem of giving a lower bound of a form into testing if there is a nowhere vanishing global section of a splitting vector bundle of $\mathbb{P}^n\backslash V_{+}(f)$. This interpretation is philosophically (not technically) similar to the interpretation given by Tevelev in \cite{tevelev2001isotropic}.
\begin{theorem}\label{main theorem-prime}
    If $p$ is a prime number, then $\mathrm{str}(\mathrm{det}_p)=p$.
\end{theorem}
\begin{proof}
    The upper bound follows from the Laplace expansion. Suppose now that $\mathrm{str}(\mathrm{det}_p)\le p-1$. By Lemma~\ref{lem-intepretion}, there is a global section of \[\bigoplus_{i=1}^{p-1}(\mathcal{O}_{\mathbb{P}(\mathrm{Mat}_p)}(1)|_{\mathbb{P}(\mathrm{Mat}_p)\backslash V_{+}(\mathrm{det}_p)})^{\otimes d_{i}}\] 
    that is nowhere vanishing on $\mathbb{P}(\mathrm{Mat}_p)\backslash V_{+}(\mathrm{det}_p)$. 
    
    From now on, we use $\mathrm{PGL}_p$ to identify $\mathbb{P}(\mathrm{Mat}_p)\backslash V_{+}(\mathrm{det}_p)$. 
    By Lemma~\ref{nonvanishing lemma}, $$c_{p-1}(\bigoplus_{i=1}^{p-1}(\mathcal{O}_{\mathbb{P}(\mathrm{Mat}_p)}(1)|_{\mathrm{PGL}_p})^{\otimes d_{i}}))=0.$$ On the other hand, by Lemma~\ref{lemma-Computational lemma of Chern class}, we have 
    \begin{align*}
        &c_{p-1}\left(\bigoplus_{i=1}^{p-1}(\mathcal{O}_{\mathbb{P}(\mathrm{Mat}_p)}(1)|_{\mathrm{PGL}_p})^{\otimes d_{i}}\right)\\
        =&\prod_{i=1}^{p-1}c_{1}\left((\mathcal{O}_{\mathbb{P}(\mathrm{Mat}_p)}(1)|_{\mathrm{PGL}_p})^{\otimes d_{i}}\right)\\
        =&\prod_{i=1}^{p-1} \left(d_{i}\cdot c_{1}(\mathcal{O}_{\mathbb{P}(\mathrm{Mat}_p)}(1)|_{\mathrm{PGL}_p})\right)\\
        =&\left(\prod_{i=1}^{p-1}d_{i}\right)\cdot c_{1}(\mathcal{O}_{\mathbb{P}(\mathrm{Mat}_p)}(1)|_{\mathrm{PGL}_p})^{p-1}.
    \end{align*}
    This is impossible by Lemma~\ref{lem-Chern class of twisting sheaf}, since $\prod_{i=1}^{p-1}d_{i}\neq 0\ (\mathrm{mod~} p)$ and $c_{1}(\mathcal{O}_{\mathbb{P}(\mathrm{Mat}_p)}(1)|_{\mathrm{PGL}_p})$ is the generator of $\mathbb{Z}[x]/\langle px,x^{p}\rangle$.
\end{proof}
It is obvious that the proof of Theorem~\ref{main theorem-prime} relies heavily on the assumption that $n$ is prime. To obtain lower bounds for general $n$, we will prove a $k$-step weak monotonicity lemma of $a_{n}\coloneqq\mathrm{str}(\mathrm{det}_n)$. To deal with field of positive characteristic, we need some notions of Hasse derivatives instead of normal derivatives, this is partly inspired by the proof of the main theorem in \cite{draisma2019topological}. 
Hasse derivatives are defined as follows: 
\begin{definition}[Hasse derivatives]
    Given a degree-$d$ polynomial $f\in\mathbb{F}[x_1,\cdots,x_m]$ and $v\in \mathbb{F}^{m}$, we write \[
f(X+tv)=\sum_{i=0}^{d}D_{v}^{i}(X)t^{i}.
\]
Then we call $D_{v}^{i}(X)$ the \textit{$i$-th $v$-directional Hasse derivative} of $f(X)$.
\end{definition}
The following is the well-known Leibniz rule for Hasse derivatives.
\begin{lemma}[Leibniz rule]\label{lem-Leibniz rule}
$D_{v}^{k}(fg)=\sum_{i=0}\limits^{k}D_{v}^{i}(f)D_{v}^{k-i}(g)$.
\end{lemma}
Next, let us compute the $k$-th $A$-directional Hasse derivative of $\mathrm{det}_{n+k}$ for certain $A$.
\begin{lemma}\label{lem-derivition of det}
    For \[
    A=\begin{pmatrix}
        I_{k} &0_{k\times n}\\
        0_{n\times k} &0_{n\times n}
    \end{pmatrix},
    \]
    we have $D_{A}^{k}(\mathrm{det}_{n+k})=\mathrm{det}_{n}$.
\end{lemma}
\begin{proof}
By expanding
\[
    \mathrm{det}(X+tA)=\mathrm{det}\begin{pmatrix}
            X_1+tI_k &X_2\\
            X_3 &X_4
        \end{pmatrix},
    \]
it is clear that the coefficient of $t^{k}$ is $\mathrm{det}(X_4)$.
\end{proof}
\begin{lemma}[$k$-step weak monotonicity]\label{k-step weak monotonicity}
    For every $n\ge 2$ and $k\in\mathbb{N}_{+}$, we have $a_{n+k}\ge \frac{1}{k+1}a_{n}$.
\end{lemma}
\begin{proof}
Assume that $\mathrm{str}(\mathrm{det}_{n+k})=r$. Then there exist $f_i,g_i$ ($1\le i\le r$) of degree less than $n+k$ such that\[
\mathrm{det}_{n+k}=\sum_{i=1}^{r}f_ig_i.
\]
Let $h_1,\cdots,h_m$ be forms in $\{f_1,\cdots,f_r,g_1,\cdots,g_r\}$ of degree at most $k$. Clearly, $m\le 2r$. Moreover, we may assume $r<\frac{n}{k+1}$. Otherwise, $r\ge\frac{n}{k+1}\ge \frac{a_n}{k+1}$ is trivial. Next, we show that there exists $A\in\mathrm{Mat}_{(n+k)\times(n+k)}$ such that \begin{enumerate}
    \item $\rank(A)=k$,
    \item $h_1,\cdots,h_m$ vanishes at $A$.
\end{enumerate}
This is simply a dimension counting. Let \(N=n+k\), and let
\(\Sigma_i\subset \mathbb P(\operatorname{Mat}_N)\) be the projective
determinantal variety of matrices of rank at most \(i\). Then
\[
\dim \Sigma_i=2Ni-i^2-1.
\]
In particular,
\[
\dim \Sigma_k-\dim \Sigma_{k-1}=2N-2k+1=2n+1.
\]
Let
\[
Y=\Sigma_k\cap V_+(h_1,\ldots,h_m).
\]
Since \(Y\) is cut out from \(\Sigma_k\) by at most \(m\) hypersurfaces,
Krull's principal ideal theorem gives
\[
\dim Y\ge \dim \Sigma_k-m.
\]
By the assumption \(r<n/(k+1)\), we have
\[
m\le 2r<\frac{2n}{k+1}<2n+1
=\dim \Sigma_k-\dim \Sigma_{k-1}.
\]
Therefore \(Y\) cannot be contained in \(\Sigma_{k-1}\). Hence there exists
a point \(A\in Y\setminus \Sigma_{k-1}\), i.e. a matrix \(A\) of rank
exactly \(k\), such that \(h_1(A)=\cdots=h_m(A)=0\).

Since $A$ has rank $k$, we can assume there exists $P$ and $Q$ in $GL_{n+k}(\mathbb{F})$ such that \[
    PAQ=\begin{pmatrix}
        I_{k} &0_{k\times n}\\
        0_{n\times k} &0_{n\times n}
    \end{pmatrix}.
    \]
    Consider a linear homomorphism $Y=PXQ$, where $X,Y$ are variable matrices. We have \[
    \mathrm{det}_{n+k}(Y)=\mathrm{det}_{n+k}(PXQ)=\mathrm{det}_{n+k}(PQ)\mathrm{det}_{n+k}(X)=\mathrm{det}_{n+k}(PQ)\sum_{i=1}^{r}f_i(X)g_i(X).
    \]
    We may assume $f_i(X)=\widetilde{f}_{i}(Y)$, $g_i(X)=\widetilde{g}_{i}(Y)$ and $h_i(X)=\widetilde{h}_{i}(Y)$. Clearly, \[
    \mathrm{str}_{Y}(\mathrm{det}_{n+k}(Y))=\mathrm{str}_{Y}(\mathrm{det}_{n+k}(X)),
    \]
    and\[
    \mathrm{det}_{n+k}(Y)=\sum_{i=1}^{r}\widetilde{f}_{i}(Y)\widetilde{g}_{i}(Y),\quad \widetilde{h}_i(\begin{pmatrix}
        I_k &0\\
        0 &0
    \end{pmatrix})=h_{i}(A)=0.
    \]
    Let $B\coloneqq\begin{pmatrix}
        I_k &0\\
        0 &0
    \end{pmatrix}$ be an $(n+k)\times(n+k)$ matrix. 
    On the one hand, $D_{B}^{k}(\mathrm{det}_{n+k})=\mathrm{det}_{n}$ by Lemma~\ref{lem-derivition of det}. 
    On the other hand, Lemma~\ref{lem-Leibniz rule} implies\[
    D_{B}^{k}\left(\sum_{i=1}^{r}f_ig_i\right)=\sum_{i=1}^{r}\sum_{j=0}^{k}D_{B}^{j}(f_i)D_{B}^{k-j}(g_i).
    \]
    This gives a decomposition of $\mathrm{det}_{n}$ with at most $(k+1)r$ terms. Note that the above is a strength decomposition since any nonzero term can only be the case $uv$, where $u,v$ are forms with positive degree; this is because any degree zero term $D^{\mathrm{deg}(h_i)}_{B}(h_i)=h_{i}(B)=0$. This completes the proof.
\end{proof}
The last ingredient is a well-known bound on the consecutive prime difference.
\begin{theorem}[\cite{baker2001difference}]\label{thm-prime gap}
    For any sufficiently large $x$, there is a prime between $[x-x^{0.525},x]$.
\end{theorem}
\begin{theorem}[A polynomial lower bound for all sizes]\label{cmain theorem-general}
For all $n\ge 2$, 
\[
  \str(\mathrm{det}_n)\ge \frac{p(n)}{n-p(n)+1}.
\]
In particular, for sufficiently large $n$, \[
  \str(\mathrm{det}_n)\ge (1-o(1))n^{0.475}.
\]
\end{theorem}
\begin{proof}
The first claim can be obtained from Theorem~\ref{main theorem-prime} and Lemma~\ref{k-step weak monotonicity}. The second claim can be deduced from the first claim and Theorem~\ref{thm-prime gap}.
\end{proof}
\begin{lemma}\label{lem-birch rank of det}
    $\mathrm{Brk}(\mathrm{det}_{n})=4$ for any $n\ge 2$.
\end{lemma}
\begin{proof}
    Note that the singular locus of $V(\mathrm{det}_{n})$ is the determinantal variety of matrices with rank at most $n-2$, which is of codimension $n^2-(n-2)(n+2)=4$.
\end{proof}
\section{Lower bound of partition rank}\label{sec:monotonicity}
Since we can identify $\mathrm{det}_n$ as an $n$-multilinear map on the column vector space, we can consider the partition rank of $\mathrm{det}_{n}$. In this section, for $n\ge 2$ we denote $b_{n}\coloneqq\mathrm{prk}(\mathrm{det}_{n})$. Clearly, by definition, $b_n\ge a_n$. Hence we have a trivial lower bound on $b_n$ according to Theorem~\ref{cmain theorem-general}. However, the lower bound for the partition rank case can be better. In what follows, we prove a stronger monotonicity lemma for partition rank. 
\begin{lemma}[Monotonicity]\label{partition rank monotonicity}
   For every $n\ge 2$, it holds that $b_{n+1}\ge b_n$.
\end{lemma}
\begin{proof}
    We may assume $b_{n+1}\le n-1$ without loss of generality, since otherwise $b_{n+1}\ge n\ge b_n$. Furthermore, we write
    \[
    \mathrm{det}_{n+1}(v_1,\cdots,v_{n+1})=\sum_{i=1}^{r}f_i((v_j)_{j\in [n]})\, g_i(v_{n+1})+\sum_{i=r+1}^{b_{n+1}}f_{i}((v_{j})_{j\in A_{i}\subsetneq[n]})\,g_{i}((v_{j})_{j\in [n+1]\backslash A_{i}}).
    \]
    Since $r\le b_{n+1}\le n-1$, we can find some $0\neq c\in\mathbb{F}^{n+1}$ such that $g_1(c)=0,\cdots,g_r(c)=0$. Then \[
    \mathrm{det}_{n+1}(v_1,\cdots,v_n,c)=\sum_{i=r+1}^{b_{n+1}}f_{i}((v_{j})_{j\in A_{i}\subsetneq[n]})\,g_{i}((v_{j})_{j\in [n]\backslash A_{i}},c).
    \]
    This shows that $\mathrm{prk}(\mathrm{det}_{n+1}(v_1,\cdots,v_n,c))\le b_{n+1}-r\le b_{n+1}$. Write \(
    c=\begin{pmatrix}
        c_1\\
        c_2\\
        \vdots\\
        c_{n+1}\\
    \end{pmatrix}
    \).  
    Without loss of generality, we may assume that $c_{n+1}\neq 0$. There exists an elementary matrix $P\in GL_{n+1}(\mathbb{F})$ such that \[
    P\cdot c=\begin{pmatrix}
        0\\
        \vdots\\
        0\\
        c_{n+1}
    \end{pmatrix},\quad \mathrm{det}_{n+1}(P)=1.
    \] 
    Let us write $P\cdot (v_1,\cdots,v_n)=\begin{pmatrix}
        X_{n\times n}\\
        Y_{1\times n}
    \end{pmatrix}$. Then we have \[
    \mathrm{det}_{n+1}(v_1,\cdots,v_n,c)=\mathrm{det}_{n+1}(P\cdot (v_1,\cdots,v_n,c))=c_{n+1}\cdot\mathrm{det}_{n}(X_{n\times n}).
    \]
    Clearly, the entries of $X_{n\times n}$ are independent variables. Therefore, we conclude that \[
    \mathrm{prk}(\mathrm{det}_n)=\mathrm{prk}(\mathrm{det}_{n+1}(v_1,\cdots,v_n,c))
    \le b_{n+1},
    \]
    which shows $b_{n+1}\ge b_n$.
\end{proof}
Combining Lemma~\ref{partition rank monotonicity} and Theorem~\ref{main theorem-prime}, we obtain the following result.
\begin{theorem}\label{main theorem-prk}
For every $n\ge 2$,
    \[
  \mathrm{prk}(\mathrm{det}_n)\ge p(n).
\]
In particular, for sufficiently large $n$, \[
  \mathrm{prk}(\mathrm{det}_n)\ge n-n^{0.525}.
\]
\end{theorem}

\section{Strength versus partition rank}\label{sec:str-vs-par}
Since we have studied the strength and partition rank of \(\det_n\), it is
natural to ask how these two quantities are related. More generally, if $T\in V_1^*\otimes\cdots\otimes V_d^*$ is a multilinear form, then every partition-rank decomposition of $T$ gives a strength decomposition of $f_T$, where $f_T$ denotes the block-multilinear homogeneous polynomial associated to $T$ on $V_1\oplus\cdots\oplus V_d$; hence,
\[
\str(f_T)\le \operatorname{prk}(T).
\]
We ask whether equality always holds:
\begin{question}\label{conj-str=prk}
     Does \(\operatorname{str}(f_T)=\operatorname{prk}(T)\) hold for every multilinear form \(T\)?
\end{question}

This can be viewed as a multilinear analogue, in spirit, of Comon's question of comparing unrestricted and symmetry-preserving tensor decompositions \cite{comon2008symmetric}. The original Comon's conjecture is false in full generality \cite{wangseigal2023lower}, but the additional multigrading present here may impose substantially stronger rigidity. Clearly, if $\str(f_{T})= \operatorname{prk}(T)$ holds, then we can obtain a better estimation of $\mathrm{str}(\mathrm{det}_n)$. However, we suspect that this equality is too strong to hold in full generality.

For the determinant, our result indicates that the strength and partition rank of $\mathrm{det}_n$ are equal for every prime number $n$. For $n=4$, Lampert and Moshkovitz \cite{lampert2026determinant} proved $$\operatorname{prk}(\mathrm{det}_4)=3.$$ 
Next, we show that the strength of $\mathrm{det}_4$ is equal to $\operatorname{prk}(\mathrm{det}_4)$.

\begin{lemma}
    $\str(\mathrm{det}_4)= 3$.
\end{lemma} 

\begin{proof}
It suffices to show that $\str(\det_4)\ge 3$. Suppose on the contrary that $$\mathrm{det}_4=g_1h_1+g_2h_2$$ with all factors homogeneous of positive degree less than 4. Then the common zero locus of $g_1,h_1,g_2,h_2$ is contained in the singular locus of $\det_4$. The singular locus is the irreducible determinantal variety \(S\) of
\(4\times 4\) matrices of rank at most \(2\), and it has codimension \(4\).
Since \(Y:=V(g_1,h_1,g_2,h_2)\) is cut out by four equations, Krull's
principal ideal theorem gives an irreducible component of \(Y\) of
codimension at most \(4\). As \(Y\subset S\), this component must be equal
to \(S\). Hence \(g_1,h_1,g_2,h_2\) all vanish on \(S\).
 However, the determinantal ideal is prime and generated by the $3\times3$ minors, and therefore has no nonzero elements of degree 1 or 2. Since in each product $g_i h_i$, at least one factor has degree at most 2, this is a contradiction. 
\end{proof}

In the remainder of this section, we prove that the slice rank satisfies the property described above.

\begin{theorem}
For every multilinear form $T$,
\[
\sr^{\rm pol}(f_T)=\sr(T).
\]
\end{theorem}

\begin{proof}
Let $T\in V_1^*\otimes\cdots\otimes V_d^*$ be a $d$-linear form. A tensor slice decomposition is immediately a polynomial slice decomposition, so
$\sr^{\rm pol}(f_T)\le \sr(T)$.

Conversely, suppose
$f_T=\sum_{a=1}^r\ell_a g_a$, where $\ell_i, g_i\in S$ are homogeneous polynomials with
$\deg \ell_i=1$ and $\deg g_i=d-1$, and set
\[
W:=\bigcap_{a=1}^r\ker \ell_a\subseteq V.
\]
Then $\codim_V W\le r$ and $f_T$ vanishes on $W$.
We now perform block Gaussian elimination with respect to
$V=V_1\oplus\cdots\oplus V_d$. 
Let $\pi_i:V_i\oplus\cdots\oplus V_d\to V_i$ be the coordinate projection. Write
\[
W_i:=W\cap (V_i\oplus\cdots\oplus V_d),
\]
 and put
$U_i:=\pi_i(W_i)$. Since
$\ker(\pi_i|_{W_i})=W_{i+1}$, we have
\[
\dim U_i=\dim W_i-\dim W_{i+1}.
\]
Hence
\begin{equation}\label{eq:dim-U_i-sum}
    \sum_{i=1}^d\dim U_i=\dim W,
\qquad
\sum_{i=1}^d\codim_{V_i}U_i=\codim_V W\le r. 
\end{equation}

It remains to show that $T$ vanishes on $U_1\times\cdots\times U_d$.
Choose an arbitrary $u_i\in U_i$ for each $i\in [d]$. By the definition of $U_i$, there is
$w_i\in W_i$ whose $i$-th component is $u_i$. Thus,
\[
w_i=(0,\dots,0,u_i,*,\dots,*),
\]
and hence the vectors $w_1,\dots,w_d$ are block upper triangular.
Consider
\[
q(t_1,\dots,t_d):=
 f_T\!\left(\sum_{i=1}^d t_iw_i\right).
\]
Every point $\sum_i t_iw_i$ lies in $W$, so $q$ vanishes on $\F^d$; in particular it vanishes on $\{0,1\}^d$.
Moreover, we have $\deg q\le d$, and the block upper triangularity gives
\begin{equation}\label{eq:coef}
    [t_1\cdots t_d]q=T(u_1,\dots,u_d). 
\end{equation}

Indeed, to obtain the square-free monomial $t_1\cdots t_d$, one must choose the diagonal term $t_i u_i$ from the $i$-th block for every $i$.
By the coefficient form of Alon's Combinatorial Nullstellensatz \cite[Lemma~2.1]{alon1999combinatorial}, a polynomial of total degree at most $d$ with nonzero coefficient of $t_1\cdots t_d$ cannot vanish on $\{0,1\}^d$. Thus \eqref{eq:coef} is zero, and therefore
$T|_{U_1\times\cdots\times U_d}=0$.

Consequently,
\[
T\in\sum_{i=1}^d
V_1^*\otimes\cdots\otimes U_i^\perp\otimes\cdots\otimes V_d^*.
\]
Expanding each $U_i^\perp$ in a basis gives a tensor slice decomposition of length at most
$\sum_i\dim U_i^\perp\le r$ by \eqref{eq:dim-U_i-sum}. Hence
$\sr(T)\le \sr^{\rm pol}(f_T)$, completing the proof.
\end{proof}

\section{Concluding remarks}\label{sec:discussion}
The results above leave several natural questions.

\subsection{What is the asymptotic behavior of strength of the determinant?} We have proved in Theorem~\ref{cmain theorem-general} a polynomial lower bound for $\str(\mathrm{det}_n)$, while the Laplace expansion gives the linear upper bound $\str(\mathrm{det}_n)\le n$.  By contrast, Theorem~\ref{main theorem-prk} determines
\[
\operatorname{prk}(\mathrm{det}_n)=n-o(n),
\]
and in particular $\operatorname{prk}(\mathrm{det}_n)=\Theta(n)$.  It would be especially interesting to decide whether
\[
\str(\mathrm{det}_n)=\Theta(n),
\]
or, more ambitiously, to determine its first-order asymptotics.

\subsection{For which $n$ does $\str(\mathrm{det}_n)=n$ hold?}  The first values exhibit a suggestive pattern:
\[
\str(\mathrm{det}_3)=\str(\mathrm{det}_4)=3,
\qquad
\str(\mathrm{det}_5)=5.
\]
Our main theorem proves $\str(\mathrm{det}_p)=p$ for every prime $p$.  Are primes the only positive integers $n\ge 2$ for which $\str(\mathrm{det}_n)=n$?  More generally, understanding how the arithmetic structure of $n$ influences the strength may lead to better knowledge both of the determinant and of the strength and partition rank. It is also interesting to determine if the Laplace expansion is the only possible (up to some trivial transformations) strength (or partition rank) decomposition for $\mathrm{det}_p$ when $p$ is prime.

\subsection{Can analytic rank and geometric rank be compared with constants independent of the tensor order?}  Chen and Ye proved in \cite{chen2024stability} that geometric rank and analytic rank are equivalent up to multiplicative constants depending on the field and the order of the tensor. Subsequent work \cite{moshkovitz2024uniform} of Moshkovitz--Zhu shows that the multiplicative constant can be independent of underlying finite field.  Related estimates with better constants, still depending on the order of the tensor, were obtained independently by Baily and Lampert in \cite{baily2024strength}.  The best known constants in this comparison still depend on the order $d$.  It is natural to ask whether there exist absolute constants $c,C>0$ such that
\[
c\,\AR(T)\le \GR(T)\le C\,\AR(T)
\]
for every tensor over every finite field, after adopting the standard conventions for these ranks.  Available calculations for natural families are consistent with an absolute comparison: for determinant tensors the two ranks are both $2$ (with a ceiling on analytic rank), while for matrix-multiplication tensors they have the same order of magnitude. 
According to a private communication of Moshkovitz and Zhu \cite{moshkovitzzhu_private}, the
analytic rank and geometric rank of the matrix-multiplication tensor differ by at most an additive constant.  Clarifying whether a dependence on $d$ is genuinely necessary would sharpen our understanding of the relation between analytic and algebro-geometric notions of randomness.

\bibliographystyle{abbrv}
\bibliography{reference}

@article{ananyan2020small,
  title={Small subalgebras of polynomial rings and {S}tillman’s conjecture},
  author={Ananyan, Tigran and Hochster, Melvin},
  journal={Journal of the American Mathematical Society},
  volume={33},
  number={1},
  pages={291--309},
  year={2020}
}

@article{baily2024strength,
  title={Strength is bounded linearly by {B}irch rank},
  author={Baily, Benjamin and Lampert, Amichai},
  journal={arXiv preprint arXiv:2410.00248},
  year={2024}
}

@article{chen2024stability,
  title={Stability of ranks under field extensions},
  author={Chen, Qiyuan and Ye, Ke},
  journal={Discrete Analysis},
  year={2025}
}

@article{moshkovitz2024uniform,
  title={Uniform stability of ranks},
  author={Moshkovitz, Guy and Zhu, Daniel G},
  journal={arXiv preprint arXiv:2411.03412},
  year={2024}
}

@article{adiprasito2021schmidt,
  title={On the {S}chmidt and analytic ranks for trilinear forms},
  author={Adiprasito, Karim and Kazhdan, David and Ziegler, Tamar},
  journal={arXiv preprint arXiv:2102.03659},
  year={2021}
}

@book{fulton2013intersection,
  title={Intersection theory},
  author={Fulton, William},
  volume={2},
  year={2013},
  publisher={Springer Science \& Business Media}
}

@article{kopparty2020geometric,
  title={Geometric rank of tensors and subrank of matrix multiplication},
  author={Kopparty, Swastik and Moshkovitz, Guy and Zuiddam, Jeroen},
  journal={Discrete Analysis},
  year={2020}
}

@article{moshkovitz2022quasi,
  title={Quasi-linear relation between partition and analytic rank},
  author={Moshkovitz, Guy and Zhu, Daniel G},
  journal={arXiv preprint arXiv:2211.05780},
  year={2022}
}

@book{hartshorne2013algebraic,
  title={Algebraic geometry},
  author={Hartshorne, Robin},
  volume={52},
  year={2013},
  publisher={Springer Science \& Business Media}
}

@article{alon1999combinatorial,
  title={Combinatorial nullstellensatz},
  author={Alon, Noga},
  journal={Combinatorics, Probability and Computing},
  volume={8},
  number={1-2},
  pages={7--29},
  year={1999},
  publisher={Cambridge University Press}
}

@article {cohen2023partition,
    AUTHOR = {Cohen, Alex and Moshkovitz, Guy},
     TITLE = {Partition and analytic rank are equivalent over large fields},
   JOURNAL = {Duke Math. J.},
  FJOURNAL = {Duke Mathematical Journal},
    VOLUME = {172},
      YEAR = {2023},
    NUMBER = {12},
     PAGES = {2433--2470},
      ISSN = {0012-7094,1547-7398},
   MRCLASS = {11B30 (15A69)},
  MRNUMBER = {4654054},
MRREVIEWER = {Ben\ Joseph\ Green},
       DOI = {10.1215/00127094-2022-0086},
       URL = {https://doi.org/10.1215/00127094-2022-0086},
}

@article {Naslund20,
    AUTHOR = {Naslund, Eric},
     TITLE = {The partition rank of a tensor and {$k$}-right corners in
              {$\Bbb F_q^n$}},
   JOURNAL = {J. Combin. Theory Ser. A},
  FJOURNAL = {Journal of Combinatorial Theory. Series A},
    VOLUME = {174},
      YEAR = {2020},
     PAGES = {105190, 25},
      ISSN = {0097-3165,1096-0899},
   MRCLASS = {11B75 (15A69)},
  MRNUMBER = {4078997},
MRREVIEWER = {Donald\ Jason\ Gibson},
       DOI = {10.1016/j.jcta.2019.105190},
       URL = {https://doi.org/10.1016/j.jcta.2019.105190},
}

@inproceedings{naslund2017upper,
  title={Upper bounds for sunflower-free sets},
  author={Naslund, Eric and Sawin, Will},
  booktitle={Forum of Mathematics, Sigma},
  volume={5},
  pages={e15},
  year={2017},
  organization={Cambridge University Press}
}

@article{kazhdan2024schmidt,
  title={Schmidt rank and singularities},
  author={Kazhdan, David and Lampert, Amichai and Polishchuk, Alexander},
  journal={Ukrainian Mathematical Journal},
  volume={75},
  number={9},
  pages={1420--1442},
  year={2024},
  publisher={Springer}
}

@article{milicevic2019polynomial,
  title={Polynomial bound for partition rank in terms of analytic rank},
  author={Mili{\'c}evi{\'c}, Luka},
  journal={Geometric and Functional Analysis},
  volume={29},
  number={5},
  pages={1503--1530},
  year={2019},
  publisher={Springer}
}

@article{janzer2019polynomial,
  title={Polynomial bound for the partition rank vs the analytic rank of tensors},
  author={Janzer, Oliver},
  journal={Discrete Analysis},
  year={2019}
}

@article{gowers2011linear,
  title={Linear forms and higher-degree uniformity for functions on {$\mathbb{F}_p^n$}},
  author={Gowers, William T and Wolf, Julia},
  journal={Geometric and Functional Analysis},
  volume={21},
  number={1},
  pages={36--69},
  year={2011},
  publisher={Springer}
}

@article{schmidt1985density,
  title={The density of integer points on homogeneous varieties},
  author={Schmidt, Wolfgang M},
  journal={Acta Mathematica},
  year={1985}
}

@inproceedings{cohen2021structure,
  title={Structure vs. randomness for bilinear maps},
  author={Cohen, Alex and Moshkovitz, Guy},
  booktitle={Proceedings of the 53rd Annual ACM SIGACT Symposium on Theory of Computing},
  pages={800--808},
  year={2021}
}

@misc{tao2016slice,
  author = {Tao, Terence},
  title = {Notes on the slice rank of tensors},
  howpublished = {What's New},
  month = {August},
  year = {2016},
  url = {https://terrytao.wordpress.com/2016/08/24/notes-on-the-slice-rank-of-tensors/},
}

@book{eisenbud20163264,
  title={3264 and all that: A second course in algebraic geometry},
  author={Eisenbud, David and Harris, Joe},
  year={2016},
  publisher={Cambridge University Press}
}

@incollection{Grothendieck1958Torsion,
  author       = {Grothendieck, Alexander},
  title        = {Torsion homologique et sections rationnelles},
  booktitle    = {Séminaire Claude Chevalley},
  volume       = {3},
  number       = {5},
  pages        = {1--29},
  year         = {1958},
  publisher    = {Secrétariat mathématique, Paris},
  series       = {Anneaux de Chow et applications},
  language     = {French},
  zbl          = {0098.13101},
  url          = {http://www.numdam.org/item/SCC_1958__3__A5_0/}
}

@misc{Mackall2018Chow,
  author       = {Mackall, Eoin},
  title        = {The Chow ring of the special linear group and {K}-(co)homology},
  year         = {2018},
  month        = jan,
  day          = {12},
  howpublished = {Blog post},
  url          = {https://eoinmackall.wordpress.com/2018/01/12/the-chow-ring-of-the-special-linear-group-and-k-cohomology/},
}

@book{Procesi2007LieGroups,
  author    = {Procesi, Claudio},
  title     = {Lie Groups: An Approach through Invariants and Representations},
  series    = {Universitext},
  publisher = {Springer},
  address   = {New York, NY},
  year      = {2007},
  edition   = {1},
  pages     = {xxii, 596},
  isbn      = {978-0-387-26040-2},
  doi       = {10.1007/978-0-387-28929-8},
  language  = {English},
  subject   = {Lie groups, Invariants, Representations of algebras}
}

@book{Milne2017Algebraic,
  author    = {Milne, J. S.},
  title     = {Algebraic Groups: The Theory of Group Schemes of Finite Type over a Field},
  publisher = {Cambridge University Press},
  year      = {2017},
  address   = {Cambridge},
  isbn      = {9781107167483},
  pages     = {660},
  language  = {English}
}

@article{Fulton1992Flags,
  author    = {Fulton, William},
  title     = {Flags, {S}chubert polynomials, degeneracy loci, and determinantal formulas},
  journal   = {Duke Mathematical Journal},
  volume    = {65},
  number    = {3},
  pages     = {381--420},
  year      = {1992},
  doi       = {10.1215/S0012-7094-92-06516-1},
  mrnumber  = {1154177},
  issn      = {0012-7094},
  zbl       = {0788.14044}
}

@misc{stacks-project,
  author       = {The {Stacks project authors}},
  title        = {Lemma 01CA: Pullback commutes with tensor product},
  howpublished = {\url{https://stacks.math.columbia.edu/tag/01CA}},
  year         = {2026},
  note         = {Tag 01CA; Chapter 17, Lemma 16.4},
}

@article{tevelev2001isotropic,
  title     = {Isotropic Subspaces of Polylinear Forms},
  author    = {Tevelev, Evgeny A.},
  journal   = {Mathematical Notes},
  volume    = {69},
  number    = {6},
  pages     = {845--852},
  year      = {2001},
  publisher = {Kluwer Academic Publishers / Plenum Publishers},
  doi       = {10.1023/A:1010294818389},
  issn      = {0001-4346}
}

@article{draisma2019topological,
  title     = {Topological {N}oetherianity of polynomial functors},
  author    = {Draisma, Jan},
  journal   = {Journal of the American Mathematical Society},
  volume    = {32},
  number    = {3},
  pages     = {691--707},
  year      = {2019},
  publisher = {American Mathematical Society},
  doi       = {10.1090/jams/923},
  issn      = {0894-0347}
}

@article{baker2001difference,
  title     = {The Difference Between Consecutive Primes, {II}},
  author    = {Baker, R. C. and Harman, G. and Pintz, J.},
  journal   = {Proceedings of the London Mathematical Society},
  volume    = {83},
  number    = {3},
  pages     = {532--562},
  year      = {2001}
}

@inproceedings{lampert2026determinant,
  author = {Lampert, Amichai and Moshkovitz, Guy},
  title = {Slice Rank and Partition Rank of the Determinant},
  booktitle = {17th Innovations in Theoretical Computer Science Conference (ITCS 2026)},
  pages = {90:1--90:15},
  series = {Leibniz International Proceedings in Informatics (LIPIcs)},
  volume = {362},
  publisher = {Schloss Dagstuhl -- Leibniz-Zentrum f{\"u}r Informatik},
  address = {Dagstuhl, Germany},
  year = {2026},
  doi = {10.4230/LIPIcs.ITCS.2026.90}
}

@article{comon2008symmetric,
  author = {Comon, Pierre and Golub, Gene and Lim, Lek-Heng and Mourrain, Bernard},
  title = {Symmetric Tensors and Symmetric Tensor Rank},
  journal = {SIAM Journal on Matrix Analysis and Applications},
  volume = {30},
  number = {3},
  pages = {1254--1279},
  year = {2008},
  doi = {10.1137/060661569}
}

@article{chenye2026geometry,
    title={Geometry of multilinear varieties over infinite fields and its applications},
    author={Qiyuan Chen and Ke Ye},
    journal={arXiv preprint},
    volume={arXiv:2605.04859},
    year={2026},
    url={https://arxiv.org/abs/2605.04859}
}

@article{wangseigal2023lower,
    title={Lower bounds on the rank and symmetric rank of real tensors},
    author={Kexin Wang and Anna Seigal},
    journal={Journal of Symbolic Computation},
    volume={118},
    pages={69--92},
    year={2023},
    month={9},
    doi={10.1016/j.jsc.2023.01.004},
    url={https://www.sciencedirect.com/science/article/pii/S0747717123000046}
}

@misc{moshkovitzzhu_private,
  author = {Guy Moshkovitz and Daniel Zhu},
  title = {Private communication: The analytic rank and geometric rank of matrix multiplication tensors differ by an additive constant},
  howpublished = {Private conversation},
  year = {2024},
  note = {Personal discussion with the author, unpublished}
}
\end{document}